\documentclass[12pt, reqno, a4, oneside]{amsart}
\usepackage[margin=1in]{geometry}

\usepackage{amssymb,latexsym,amsmath,amsthm,enumitem,mathrsfs}
\usepackage{xcolor,float}

\usepackage{hyperref}
\hypersetup{colorlinks=true,linkcolor=blue,citecolor=blue}

\title[Connection between low-lying zeros and central values]{A connection between low-lying zeros and central values of $L$-functions}

\author[Lesesvre]{Didier Lesesvre}
\address{CNRS – Université de Montréal CRM – CNRS \newline \indent Université de Lille -- Laboratoire Paul Painlevé, UMR 8524,
 59000 Lille, France}
\email{didier.lesesvre@univ-lille.fr}

\author[Suriajaya]{Ade Irma Suriajaya}
\address{Department of Mathematics and Statistics, San Jose State University
\newline \indent Faculty of Mathematics, Kyushu University, 744 Motooka, Nishi-ku \newline \indent Fukuoka 819-0395, Japan}
\email{adeirmasuriajaya@math.kyushu-u.ac.jp}

\numberwithin{equation}{section}

\newtheorem{conj}{Conjecture}
\newtheorem{prop}{Proposition}
\newtheorem{lem}{Lemma}
\newtheorem{thm}{Theorem}
\newtheorem{coro}{Corollary}
\newtheorem{hyp}{Hypothesis}
\theoremstyle{remark}
\newtheorem{rk}{Remark}

\newcommand{\GL}{\mathrm{GL}}

\date\today

\begin{document}

\begin{abstract}
We discuss the relation between statistics on low-lying zeros of $L$-functions and distribution of the associated central values.
More precisely, we deduce explicit conditional lower bounds toward the Keating-Snaith conjecture (on the distribution of central values of families of $L$-functions) from partial results toward the Rudnick-Sarnak density conjecture (on the one-level density for the low-lying zeros of these $L$-functions). We show in fact that the same crucial ingredient occurs in the classical approaches for proving both results, providing the connection. We precisely determine the relation between the type of symmetry of the family, the allowed Fourier support in its distributional statement, and the quality of the lower bounds obtained. 
\end{abstract}

\maketitle

This paper is the general pendant of \cite{RS24} (for quadratic twists of elliptic curves) and~\cite{LS} (for modular forms), proving analogous results in the setting of the Selberg class. It takes for granted properties of $L$-functions that are expected but out of reach or very difficult to prove in specific cases, in order to emphasize the general and robust ideas behind the proofs, as announced in \cite{RS24}.

\section{Introduction}
\label{sec:introduction}

\subsection{Families of \texorpdfstring{$L$}{L}-functions}
\label{subsec:intro-families}

Objects of very different nature have attached $L$-functions each of which can be seen as a generating function $L(s)$, with a corresponding conductor~$c(L)$ representing its complexity. This comprises the cases of sequences of integers, Dirichlet characters, number fields, Galois representations, elliptic curves, abelian varieties, geodesics on hyperbolic surfaces, etc. as explained in \cite{IS00-Lfun}. The functoriality principle postulates that all these $L$-functions arise as automorphic $L$-functions attached to automorphic representations on general linear groups~$\mathrm{GL}(n)$.

We consider the very general framework of \textit{families} $\mathcal{F}$ of $L$-functions in the sense of Sarnak: geometric families of automorphic $L$-functions, families of Artin $L$-functions, universal families of $L$-functions of automorphic representations on $\mathrm{GL}(n)$, or other families arising from these under number theoretic constraints, functorial transfers, spectral restrictions or algebraic constructions. See \cite{SST} for a more precise discussion. The family $\mathcal{F}$ is generally infinite but we assume it can be truncated to a finite family~$\mathcal{F}_X$ when restricted to those elements of size $c(L) \leqslant X$, for any $X>0$. Such families include all classical families such as Dirichlet characters with bounded conductor, elliptic curves of bounded rank, automorphic representations with prescribed ramification, modular forms with bounded weight or level, Maass forms with bounded eigenvalue, Rankin-Selberg $L$-functions, etc. All the properties effectively needed for $L$-functions in the family will be detailed in Section \ref{sec:$L$-functions}.

\subsection{Distribution of zeros and the density conjecture}
\label{subsec:intro-distribution-zeros}

The zeros of an $L$-function $L(s)$ carry important information about the underlying arithmetical object. Obtaining precise information about their distribution, of which the Grand Riemann Hypothesis is the typical expectation, has striking consequences, see \cite[pp. 712--713]{IS00-Lfun}.

The spacings of zeros of families of $L$-functions are well-understood; they are distributed along a universal law, independent of the exact family at hand, as proven by Rudnick and Sarnak \cite{RudSar94}. This recovers the behavior of spacings between eigenangles of the classical groups of random matrices.
It can however be expected that distribution of \textit{low-lying} zeros attached to every reasonable family of $L$-functions also behaves as the analogous distribution of the small eigenangles of the classical groups of random matrices, and this depends upon the specific setting under consideration, providing non-trivial information about the objects. This gives rise to the notion of \textit{type of symmetry} of a family.

More precisely, let $L(s)$ be an $L$-function (see Section \ref{sec:$L$-functions}). Consider its nontrivial zeros written in the form
\begin{equation}
\label{def:rho}
\rho = \frac{1}{2}+i\gamma, 
\end{equation}
where $\gamma$ is a priori a complex number.
We renormalize the mean spacing of the zeros to $1$ by setting
\begin{equation}
\label{def:gamma}
\tilde{\gamma} := \frac{\log c(L)}{2\pi} \gamma.
\end{equation}
Let $\phi$ be an even Schwartz function on $\mathbb{R}$ whose Fourier transform is compactly supported, in particular it admits an analytic continuation to all $\mathbb{C}$. The one-level density attached to~$L$ is defined by
\begin{equation}
\label{def:one-level-density}
D(L, \phi) := \sum_{\gamma} \phi\left(\tilde{\gamma}\right), 
\end{equation}
where the summation runs over the (imaginary parts of the) zeros of $L(s)$.
A wide literature has been published concerning the statistical behavior of low-lying zeros of families of $L$-functions, see \cite{SST} for a summary of various results or Table \ref{table:known-consequences} below. The analogy with the behavior or small eigenangles of random matrices led Katz and Sarnak \cite{katz_zeroes_1999} to formulate the so-called \emph{density conjecture}, claiming the same universality for the types of symmetry of families of $L$-functions as those arising for classical groups of random matrices.

\begin{conj}[Katz-Sarnak] \label{conj:katz-sarnak}
Let $\mathcal{F}$ be a family of $L$-functions, and $\mathcal{F}_X$ a finite truncation increasing to $\mathcal{F}$ when $X$ grows. Then for all even Schwartz function $\phi$ on $\mathbb{R}$ with compactly supported Fourier transform, there is one classical group $G$ among $\mathrm{U}$, $\mathrm{SO(even)}$, $\mathrm{SO(odd)}$, $\mathrm{O}$ or $\mathrm{Sp}$ such that
\begin{equation}
\label{density-conjecture}
\frac{1}{|\mathcal{F}_X|} \sum_{f \in \mathcal{F}_X} D(L, \phi) \xrightarrow[X\to\infty]{} \int_{\mathbb{R}} \phi(x) W_G(x)dx, 
\end{equation}
where $W_G(x)$ is the explicit density function modeling the distribution of the eigenangles of the corresponding group of random matrices. More precisely, for all $x \in \mathbb{R}$, 
\begin{align}
    W_{\rm U}(x) & = 1, \\
    W_{\rm O}(x) & = 1 + \frac{1}{2}\delta_0(x), \\
    W_{\rm SO(even)}(x) & = 1 + \frac{\sin 2\pi x}{2\pi x}, \\
    W_{\rm SO(odd)}(x) & = 1 - \frac{\sin 2\pi x}{2\pi x} + \frac{1}{2}\delta_0(x), \\
    W_{\rm Sp}(x) & = 1 - \frac{\sin 2\pi x}{2\pi x}, 
\end{align}
where $\delta_0$ is the Dirac distribution at $0$.
The family $\mathcal{F}$ is then said to have the type of symmetry~$G$.
\end{conj}

\begin{rk}
For families of $L$-functions associated to algebraic varieties over function fields, the type of symmetry is determined by the monodromy of the family, shedding light on the reason why zeros of $L$-functions are governed by groups of random matrices, see \cite{katz_zeroes_1999}. No such analogue is known for number fields, and no case of Conjecture \ref{conj:katz-sarnak} has been proven yet, even though restricted versions exist in various cases (see Table \ref{table:known-consequences}).
\end{rk}

\subsection{Distribution of central values and the Keating-Snaith conjecture}
\label{subsec:intro-distrib-central-values}

The values of the $L$-functions at special points are also of considerable importance, and appear to carry a lot of information, see \cite{IS00-Lfun}. For instance, in the automorphic forms setting, the non-vanishing at the central point of the $L$-function attached to a Rankin-Selberg product $\pi_1 \times \pi_2$ is conjectured to be related to the non-vanishing of automorphic periods of arithmetic significance, see \cite{GJR}. With a geometric flavor, given an elliptic curve, the Birch and Swinnerton-Dyer conjecture relates the order of vanishing of its associated $L$-function at the central point to the rank of that elliptic curve. One example coming from algebraic number theory states that special values of the $L$-function attached to some characters are related to important invariants of the underlying field extension, by means of the class number formula. The study of the distribution of such values is therefore of natural significance. 

The Keating-Snaith conjecture \cite{KeaSna00-L} predicts that the logarithmic central values $\log L(\tfrac12)$ are asymptotically distributed according to a normal distribution, with explicit mean and variance depending on the family.
 
\begin{conj}[Keating-Snaith] \label{conj:keating-snaith}
Let $\mathcal{F}$ be a family of $L$-functions of type of symmetry $\mathrm{SO(even)}$ or $\mathrm{Sp}$, and $\mathcal{F}_X$ a finite truncation increasing to $\mathcal{F}$ when $X$ grows. Then there exists~$M_{\mathcal{F}_X}$ and $V_{\mathcal{F}_X}$ such that, for any positive real numbers $\alpha < \beta$,
\begin{equation}
\label{keating-snaith-conjecture}
\frac{1}{|\mathcal{F}_X|}\left| \left\{ f \in \mathcal{F}_X \ : \ \frac{\log L(\tfrac12) - M_{\mathcal{F}_X}}{V_{\mathcal{F}_X}} \in (\alpha, \beta) \right\} \right| \xrightarrow[X\to\infty]{} \frac{1}{\sqrt{2\pi}} \int_\alpha^\beta e^{-x^2/2} dx,
\end{equation}
when $X$ grows to infinity. In other words, the family of the logarithms of the central values~$\log L(\tfrac12)$ equidistributes asymptotically with respect to a normal distribution with mean $M_{\mathcal{F}_X}$ and variance $V_{\mathcal{F}_X}$.
\end{conj}
When $L(\tfrac12)$ vanishes or is negative, then $\log L(\tfrac12)$ is considered to be $-\infty$ so that the condition in the above set is not satisfied and the corresponding $L$-function does not contribute to the proportion.

The only unconditional instance of such a result is the equidistribution of the family of $\log|\zeta(\tfrac12+it)|$ when $t$ varies, as proven by Selberg (unpublished but stated for general $L$-functions in \cite[Theorem 2]{Selberg89}, we refer the reader to the proof given in \cite[Theorem 6.1]{Tsang_thesis} and \cite{radziwill_selbergs_2018}, recently extended to more general cases \cite{CLW22_CLT} and also in a weighted version \cite{BELP25_wCLT}.
Note that in the vertical aspect, i.e. for vertical shifts $L(\tfrac12+it)$ of a given $L$-function $L(s)$ satisfying similar hypotheses to our class of $L$-functions defined in Section~\ref{sec:$L$-functions}, Bombieri and Hejhal \cite{BH} proved that Conjecture \ref{conj:keating-snaith} holds under the Generalized Riemann Hypothesis and strong zero-density results.

\section{\texorpdfstring{$L$}{L}-functions}
\label{sec:$L$-functions}

We introduce in this section a general class of $L$-functions, in the spirit of the Selberg class in line with \cite{BH}. To lighten notation, we denote the growing family by $\mathcal{F}$ instead of $\mathcal{F}_X$ and write $|\mathcal{F}|\to \infty$ to mean $X\to \infty$. We denote $c(\mathcal{F})$ the maximum of the conductors $c(L)$ for $L \in \mathcal{F}$.

\subsection{Definition of \texorpdfstring{$L$}{L}-functions}
\label{subsec:$L$-functions}

An $L$-function $L(s)$ is a Dirichlet series of the form
\begin{equation}
\label{dirichlet-series}
    L(s) = \sum_{n=1}^\infty \frac{a_L(n)}{n^s}
\end{equation}
for certain coefficients $a_L(n) \in \mathbb{R}$ called the coefficients of the $L$-function. The assumption that the coefficients are real is not important for the method, and general complex coefficients could be taken as the cost of putting conjugate in the orthogonality relations (Hypothesis~\ref{hyp:p-orthogonality}) and using the complex rather than real distribution in Proposition \ref{prop3}, as in \cite{BH}.
We impose this condition only to simplify the presentation.

We assume that this Dirichlet series converges (locally uniformly) on the right-half plane $\Re(s) > 1$; it is then holomorphic on this half-plane. We assume that this $L$-function admits an Euler product representation of the form
\begin{equation}
\label{euler-product}
    \prod_p \prod_{i=1}^d (1-\alpha_{L,i}(p)p^{-s})^{-1}
\end{equation}
when $\Re(s) > 1$, where the outer product runs over primes, the $\alpha_{L,i}(p)$ are complex numbers called the spectral parameters of $L(s)$, and $d \in \mathbb{N}$ is called the degree of the $L$-function. The Euler product \eqref{euler-product} in particular implies that the coefficients are multiplicative, i.e. that $a_L(nm) = a_L(n) a_L(m)$ if $n$ and $m$ are coprime integers. We introduce the log-derivative
\begin{equation}
    \frac{L'}{L}(s) = - \sum_{n \geqslant 1} \frac{\Lambda_L(n)}{n^s} \qquad \text{where} \ \Lambda_L(n) :=
    \begin{cases}
        \sum_{i=1}^d \alpha_{L,i}(p)^k \log p & \text{if } n=p^k, \\
        0 & \text{otherwise.}
    \end{cases}
\end{equation}

We moreover assume that there is an Archimedean factor
\begin{equation}
\label{def:gamma-factor}
    \gamma_L(s) = q_L^s \prod_{i=1}^d \Gamma\left( \lambda_{L,i}s + \mu_{L,i}\right) 
\end{equation}
where $q_L$ is an integer, $\lambda_{L,i} > 0$ and $\Re(\mu_{L,i}) \geqslant 0$, such that the completed $L$-function defined by $\Lambda_L(s) := \gamma_L(s) L(s)$ admits a meromorphic continuation to the whole complex plane, with a finite number of poles, all on the line $\Re(s)=1$, and satisfies the functional equation
\begin{equation}
\label{functional-equation}
    \xi_L(s) = \varepsilon_L \overline{\xi_L}(1-s)
\end{equation}
where $\overline{\xi}_L(s) = \overline{\xi_L(\bar{s})}$ and $\varepsilon_L$ is a complex number of modulus one called the sign of the functional equation.
The integer $q_L$ is called the \textit{arithmetic conductor} of $L$, and we define the \textit{analytic conductor} of~$L$ by 
\begin{equation}
\label{conductor}
c(L) := q_L \prod_{i=1}^d (1+|\mu_{L,i}|)
\end{equation}
which encapsulates the complexity of $L(s)$. We introduce in the remaining of this section further properties of $L$-functions that are much expected, but out of reach or only partially known in some specific settings.

\subsection{Bound on coefficients}

We need bounds on the spectral parameters and on the coefficients of the $L$-functions.

\begin{hyp}[Ramanujan conjecture] 
\label{Hyp:bound-spectral-parameters}
We assume that for each prime $p$ we have the bound, for all $\varepsilon > 0$ and all $1 \leqslant i \leqslant d$, 
\begin{equation}
\label{hyp:bound-spectral-parameters}
|\alpha_{L,i}(p)| \ll p^{\varepsilon}.
\end{equation}
 The implied constant does not depend on $p$ or on $L$.
\end{hyp}
This assumption has a strong interpretation in terms of the underlying object (e.g. for automorphic $L$-functions, it essentially states that the underlying automorphic representation is tempered). 
We can derive from \eqref{hyp:bound-spectral-parameters} that $a_L(n) \ll n^\varepsilon$.
We in particular deduce the following Rankin-type estimate:
\begin{equation}
    \sum_{p\leqslant X} \sum_{i=1}^d |\alpha_{L,i}(p)|^2 \ll X^{1+\varepsilon}.
\end{equation}

\subsection{Orthogonality in \texorpdfstring{$p$}{p}}
\label{subsec:hyp-bounds-spectral-parameters}

We assume the following Selberg orthogonality conjecture.
\begin{hyp}[Orthogonality in $p$] 
\label{hyp:p-orthogonality}
We have, for all $X \geqslant 1$ and two $L$-functions $L$ and $M$ in $\mathcal{F}$, 
\begin{equation}
    \sum_{p \leqslant X} \frac{a_L(p)\overline{a_M(p)}}{p} = \delta_{L=M}n_{\mathcal{F}}\log \log X + O(1), 
\end{equation}
for a certain constant $n_{\mathcal{F}} > 0$, and where $\delta_{L=M}$ is the characteristic function of $L=M$.
\end{hyp}

Note that assuming such an orthogonality relation implies in particular that the $L$-functions are primitive in the sense of Selberg, i.e. that they do not have any nontrivial common factors. We omit the precise definitions of these terminologies since Hypothesis~\ref{hyp:p-orthogonality} already includes the properties required in our method. We refer the readers to \cite{BH}, for instance, for further details.

\subsection{Hecke relations}

We assume a certain relation between the spectral parameters and the coefficients of $L$-functions, which can equivalently be phrased as the existence of an Euler product representation with a specific shape.

\begin{hyp}[Hecke relations] 
\label{Hyp:hecke-relations}
There exist constants $\gamma_{\mathcal{F}}$ and $\gamma_{1}$ such that, for all $L \in \mathcal{F}$, 
\begin{equation}
    \sum_{i=1}^d \alpha_{L,i}(p)^2 = a_L(p^2) + \gamma_1 a_L(p) + \gamma_{\mathcal{F}}.
\end{equation}
\end{hyp}

By identification of terms between the Dirichlet series and the Euler product expression of $L(s)$, we have that the sum of the $\alpha_{L,i}(p)$ is the coefficient $a_L(p)$. The strength of the above hypothesis is that it assumes that $\gamma_{\mathcal{F}}$ and $\gamma_1$ are constants depending only on $\mathcal{F}$, but not on $L\in \mathcal{F}$. The term~$\gamma_{\mathcal{F}}$ is the critical ``constant term" contributing to the type of symmetry of the family and to the mean of the logarithms of central values (see Theorem \ref{thm}). This assumption is made for simplicity and does not hold in every family: it suffices in practice to ensure that it is true on average over the family, since all the results stated in this paper are on average, but this choice has been made for ease of exposition.

\subsection{Orthogonality in \texorpdfstring{$\mathcal{F}$}{F}}

We assume the following orthogonality relation in $\mathcal{F}$: 
\begin{hyp}[Orthogonality in $\mathcal{F}$] \label{Hyp:tf2}
There exists $\gamma >0$ such that, for all $n \geqslant 2$, 
\begin{equation}
\label{hyp:tf2}
\frac{1}{|\mathcal{F}|} \sum_{L\in \mathcal{F}} a_L(n)a_L(m) = \delta_{n=m} + O(|\mathcal{F}|^{-\gamma}).
\end{equation}
\end{hyp}
The above hypothesis quantifies the fact that there are oscillations in the coefficients $a_L(n)$ when it is not trivial (which we consider here to be the $n=1$ case only for simplicity, but in practice it may occur for other small set of exceptional values, e.g. $n$ being a square when~$a_L$ is a quadratic character, etc.), and that there are compensations incurring a power saving compared to the main term in \eqref{hyp:tf2}. It stems in practice from trace formulas, such as Poisson summation formula for characters, Petersson trace formula for modular forms, etc.

From the above orthogonality statement, we deduce a mean value theorem, analogous to \cite[Lemma 4]{BH}.
\begin{coro}[Mean value theorem]
\label{coro:mean-value-theorem}
    For all $x < |\mathcal{F}|^{\gamma/2 - \varepsilon}$ for a certain $\varepsilon>0$, we have 
    \begin{equation}
        \frac{1}{|\mathcal{F}|}\sum_{L \in \mathcal{F}} \left( \sum_{n<x} a_L(n) \right) \left( \sum_{m<x} a_L(m) \right) = \frac{1}{|\mathcal{F}|}\sum_{L \in \mathcal{F}} \sum_{n<x} a_L(n)^2 + o(1). 
    \end{equation}
\end{coro}
This result, and its iterations, will be useful to study moments of sums of coefficients, as will be critical below.

\subsection{Off-diagonal contributions}

Our method requires to finely understand averages of coefficients over $L\in \mathcal{F}$ and over primes $p$ simultaneously. While the orthogonality assumption of Hypothesis \ref{Hyp:tf2} controls such averages when the summation over primes $p$ is small enough, we need to explain the behavior of longer sums of coefficients. This is the content of the following assumption.

\begin{hyp}[Beyond orthogonality] \label{Hyp:tf}
There exists $\delta > 0$ such that, for a compactly supported function $\phi$ with support in $(-\delta, \delta)$, we have
\begin{equation}
\label{eq:orthogonality}
    \frac{2}{|\mathcal{F}|}\sum_{L \in \mathcal{F}} \sum_p \frac{a_L(p^k)}{p^{k/2}} \widehat{\phi}\left( \frac{k\log p}{\log c(\mathcal{F})} \right) \frac{\log p}{\log c(\mathcal{F})} =
    \begin{cases}
        o(1) & \text{if } k = 2, \\
        H_\phi + o(1) & \text{if } k=1, 
    \end{cases}
\end{equation}
where $H_\phi$ is a continuous linear functional in $\phi$, which can be written as $H_\phi = \int h_{\mathcal{F}} \phi$ for a certain function $h_{\mathcal{F}}$. We have $H_\phi = 0$ if $\delta < 1$.
\end{hyp}

The mixed orthogonality relation (in $L$ and $p$) is assumed for a single $\delta >0$, and this corresponds to the restriction toward the density conjecture as always displayed in existing results.
Applying the hypothesis for $\delta \leqslant 1$, we deduce
\begin{align}
\label{coarse}
    \frac{1}{|\mathcal{F}|}\sum_{L \in \mathcal{F}} \sum_p \frac{a_L(p)}{p^{1/2}} w(p) & =o(1),\\ \label{coarse-2}
     \frac{1}{|\mathcal{F}|}\sum_{L \in \mathcal{F}} \sum_p \frac{a_L(p^2)}{p} w(p^2) & =o(1),
\end{align}
where $w(p)$ is any bounded function in $p$ truncating the summation to $p \ll c(L)$. However, when the summation over $p$ is very long (compared to the family), viz. $\delta > 1$, the sum may get too large; this is why we require the more precise assumption above. The discrepancy $H_\phi$ is fundamental since it may occur in practice and it does contribute to the type of symmetry of certain families, see for instance \cite[Section 8]{ILS}.

We also deduce from the Hecke relations in Hypothesis \ref{Hyp:hecke-relations} and from the orthogonality statement in Hypothesis \ref{Hyp:tf} the following asymptotic estimate:
\begin{equation}
    \label{eq:rankin-selberg-on-average}
    \frac{1}{|\mathcal{F}|}\sum_{L\in \mathcal{F}} \sum_{p<x} \frac{\Lambda_L(p^2)}{p \log p} = \frac{1}{|\mathcal{F}|}\sum_{L\in \mathcal{F}} \sum_{p<x} \frac{a_L(p^2) + \gamma_1 a_L(p)}{p} + \gamma_{\mathcal{F}} \sum_{p<x} \frac{1}{p} \sim \gamma_{\mathcal{F}} \log \log x, 
\end{equation}
for all $x>2$, since the average over coefficients $a_L(p)$ and $a_L(p^2)$ is negligible by \eqref{coarse} and~\eqref{coarse-2}. This can be seen as a Rankin-Selberg statement on average; for instance it is known for modular forms as in \cite[(4.3)]{RS24} without averaging.

\subsection{Main result}

Under the assumptions of Section \ref{sec:$L$-functions}, we deduce partial results toward both Conjectures \ref{conj:katz-sarnak} and \ref{conj:keating-snaith}. We remark that the proof of both statements in our main theorem below, rely on the same techniques (and indeed, all the assumptions of Section \ref{sec:$L$-functions} are made so that these techniques work), but the statement on low-lying zeros is used to prove the statement on central values, hence shedding light on the statement made in \cite{RS24} that results on low-lying zeros lead to results on central values. Explaining these connections is the main purpose of this paper.

\begin{thm} \label{thm}
Let $\mathcal{F}$ be a family of $L$-functions as defined in Section \ref{sec:$L$-functions}, and assume the Generalized Riemann Hypothesis. Then
\begin{enumerate}
    \item (Low-lying zeros) For every Schwartz function $\phi$ such that $\widehat{\phi}$ is compactly supported in $(-\delta, \delta)$, we have
    \begin{equation}\label{Hyp:weighted-llz}
        \frac{1}{|\mathcal{F}|} \sum_{L \in \mathcal{F}} D(L, \phi) \longrightarrow \int_{\mathbb{R}} \phi(x) W_{\mathcal{F}}(x) dx, 
    \end{equation}
    where $W_{\mathcal{F}} = 1 - \tfrac{1}{2}\gamma_{\mathcal{F}}\delta_0 + h_{\mathcal{F}}$ and $\delta_0$ represents the Dirac mass at $0$. Moreover, the proportion of $L\in \mathcal{F}$ such that $L(\tfrac12)$ does not vanish is at least $\eta(\mathcal{F},\delta)$, which only depends on the type of symmetry of $\mathcal{F}$ and the value of $\delta$. Values of $\eta(\mathcal{F},\delta)$ are given in Table~\ref{table:eta}.
    \\
    \item (Central values) For the same proportion~$\eta(\mathcal{F}, \delta) \in [0,1]$, we have
\begin{equation}
\label{eq:thm}
\frac{1}{|\mathcal{F}|}\left| \left\{ f \in \mathcal{F} \ : \ \frac{\log L(\tfrac12) - M_{\mathcal{F}}}{V_{\mathcal{F}}} \in (\alpha, \beta) \right\} \right| \geqslant \frac{\eta(\mathcal{F}, \delta)}{\sqrt{2\pi}} \int_\alpha^\beta e^{-x^2/2} dx, 
\end{equation}
where $M_{\mathcal{F}} = \tfrac12 \gamma_{\mathcal{F}} n_{\mathcal{F}}\log \log c(\mathcal{F})$ and $V_{\mathcal{F}} = \sqrt{n_{\mathcal{F}}\log \log c(\mathcal{F})}$. 
\begin{center}
\begin{table}[ht]
\begin{tabular}{|c||c|c|c|c|c|c|}
 \hline
  & \multicolumn{3}{| c |}{{Any} $\varepsilon_L$} & \multicolumn{3}{| c |}{$\varepsilon_L = 1$} \\
  \hline
 \emph{Types} & $\delta <1$ & $\delta \geqslant 1$ & $\delta = \infty$ & $\delta <1$ & $\delta \geqslant 1$ & $\delta = \infty$ \\
 \hline\hline
 $\mathrm{O}$ & $\tfrac{1}{2} - \tfrac{1}{\delta }$ & $\tfrac{1}{2} - \tfrac{1}{\delta }$ & $\tfrac12$ & $\tfrac{3}{4} - \tfrac{1}{2\delta }$ & $\tfrac{3}{4} - \tfrac{1}{2\delta }$ & $\tfrac34$ \\
 \hline
  $\mathrm{SO(even)}$ & $\tfrac{1}{2} - \tfrac{1}{\delta }$ & $1 - \tfrac{2}{\delta } +\tfrac{1}{2\delta ^2}$ & $1$ & $\tfrac{3}{4} - \tfrac{1}{2\delta }$ & $1 - \tfrac{1}{\delta } +\tfrac{1}{4\delta ^2}$ & $1$ \\
  \hline
  $\mathrm{SO(odd)}$ & $\tfrac{1}{2} - \tfrac{1}{\delta }$ & $- \tfrac{1}{2\delta ^2}$ & $0$ & $\tfrac{3}{4} - \tfrac{1}{2\delta }$ & $\tfrac12 - \tfrac{1}{4\delta ^2}$ & $\tfrac12$ \\
 \hline
    $\mathrm{Sp}$ & $\tfrac{3}{2} - \tfrac{1}{\delta }$ & $1 - \tfrac{1}{2\delta ^2}$ & $1$ & $\tfrac{5}{4} - \tfrac{1}{2\delta }$ & $1 - \tfrac{1}{4\delta ^2}$ & $1$ \\ 
 \hline
    $\mathrm{U}$ & $1-\tfrac{1}{\delta }$ & $ 1 - \tfrac{1}{\delta }$ & $1$ & $1-\tfrac{1}{2\delta }$ & $ 1 - \tfrac{1}{2\delta }$ & $1$ \\
 \hline
\end{tabular}
\caption{Description of $\eta(\mathcal{F},\delta)$, emphasizing cases where $\varepsilon_L=1$ for all $L \in \mathcal{F}$}
\label{table:eta}
\end{table}
\end{center}
\end{enumerate}
\end{thm}

Such a result has been recently obtained by Radziwi{\l}{\l} and Soundararajan~\cite{RS24} in the case of quadratic twists of an elliptic curve. Their achievement inspired this work and the ideas and methods are ultimately already present in their article. The driving force of the present paper is to explain the interaction between  low-lying zeros and  distribution of central values, which benefits from the lights of a general framework. It has to be emphasized, as also noted in \cite{RS24}, that it is the same density $\eta(\mathcal{F}, \delta)$ that occurs in the non-vanishing result and in the lower bound toward Conjecture \ref{conj:keating-snaith}. 

Note that the proportion in Conjecture \ref{conj:keating-snaith} is always expected to be upper-bounded by the conjectured limit, as mentioned in \cite{radziwill_moments_2015}, so that Theorem \ref{thm} provides complementary information on the lower bounds.

\begin{rk}
In the case of families where all the signs are known to be $\varepsilon_L=1$, the arguments can be improved to obtain better bounds as in \cite{ILS} or \cite{RS24}, justifying that the cases are written separately --- see Section \ref{sec:l-values-and-zeros}.
In the case of unitary, symplectic or even orthogonal types of symmetry when the signs are always $\varepsilon_L=1$, then $\eta(\mathcal{F}, \delta)$ tends to $1$ as $ \delta $ grows, hence the density conjecture implies the expected lower bound of the distribution conjecture. However, in the case of the odd orthogonal type of symmetry, where the signs are never $1$, all the central values vanish at $\tfrac12$ and therefore even $\delta = \infty$ cannot give better than $\eta(\mathcal{F}, \infty) = 0$. The two cases of orthogonal or odd orthogonal type of symmetry are not expected to allow $\varepsilon_L = 1$ for all $L$-functions, so that the apparent excessive density they incur is not contradictory.
\end{rk}

\begin{rk}
This result may lead to trivial results, i.e. $\eta(\mathcal{F}, \delta) = 0$, when the support $ \delta $ is too small, even when it is an explicit positive constant. A non-trivial lower bound toward the distribution of central values is obtained when the support limit $ \delta $ is larger than a minimal value $ \delta _{\min}$ given by the following table.
\begin{center}
\begin{table}[ht]
\begin{tabular}{|c||c|c|}
\hline
& \multicolumn{2}{|c|}{$\delta_{\min}$} \\
\hline
Types of symmetry & Any $\varepsilon$ & $\varepsilon = 1$ \\
\hline\hline
O & 2 & 2/3 \\
\hline
SO(even) & $1+\tfrac{\sqrt{2}}{2}$ & 2/3 \\
\hline
SO(odd) & $\varnothing$ & 2/3 \\
\hline
Sp & 2/3 & 2/5\\
\hline
U & 1 & 1/2 \\
\hline
\end{tabular}
\caption{Lowest Fourier support $ \delta _{\rm min}$ giving a nontrivial bound $\eta(\mathcal{F}, \delta)>0$}
\end{table}
\end{center}
\end{rk}

\subsection{Known results and consequences}
\label{subsec:intro-known-results}

While the general framework of Section \ref{sec:$L$-functions} aims at explaining and emphasizing the relations between the low-lying zeros and the distribution of $L$-values, as claimed in \cite{RS24} and instantiated therein for the case of quadratic twists of an elliptic curve, or in \cite{LS} in the case of modular forms, various results toward the density conjecture of Rudnick and Sarnak are actually known.
Therefore, as a consequence we can deduce analogous explicit lower bounds toward the Keating-Snaith conjecture, with explicit values for $\eta(\mathcal{F}, \delta)$. We summarize some of these results and expected consequences below, without pretending to any exhaustivity in this very active area, but underlining the diversity of the settings.
\begin{table}[ht]
\begin{center}
\begin{tabular}{| c | c | c | c | c | c |} 
 \hline
 $\mathcal{F}$ & $c(\mathcal{F})$ & Type & $\delta$ & $\eta(\mathcal{F}, \delta)$ & References \\ [0.5ex] 
 \hline
Twists $ E \times \chi_d$ ($\varepsilon=1$) & $d$ & SO(even) & 1 & 1/4 & \cite{ozluk_small_1993} \\ 
Dirichlet $L$-functions & conductor & Sp & 2 & 7/8 &\cite{ozluk_small_1993} \\ 
Modular forms ($\varepsilon=1$) & level & SO(even) & 2 & 9/16 & \cite{ILS} \\ 
$\mathrm{Sym}^2$ Modular forms & level & Sp & 3/2 & 8/9 & \cite{ILS} \\ 
Maaß forms & spectral & U & 2 & 1/2 & \cite{alpoge_low-lying_2014} \\ 
GL(3) & spectral & U & 10/13 & 0 &\cite{blomer_applications_2013} \\ 
GL(3) $\mathrm{Sym^2}$ & spectral & U & 5/23 & 0 &\cite{blomer_applications_2013} \\ 
GL(3) Adjoint & spectral & Sp & 2/27 & 0 & \cite{blomer_applications_2013} \\
GSp(4) modular Spin & weight & O & 1 & 0 & \cite{zhao} \\ 
GSp(4) modular Std & weight & Sp & 1 & 1/2 & \cite{kwy} \\ 
GSp(4) Maass Spin & spectral & O & \cite{comtat_effective_2025} & 0 & \cite{comtat_effective_2025} \\
GSp(4) Maass Std & spectral & Sp & \cite{comtat_effective_2025} & 0 & \cite{comtat_effective_2025} \\
Maaß $\times$ modular & spectral & O & $\varepsilon$ & 0 & \cite{duenez_low_2006} \\ 
Maaß $\times$ $\mathrm{Sym^2}$ modular & spectral & O & 5/24 & 0 & \cite{duenez_low_2006} \\ 
Quaternion algebra & conductor & O & 2/3 & 0 & \cite{MR4398244} \\ 
Hilbert modular forms & level & O & 3/2 & 5/12 & \cite{liu_low-lying_2017} \\ 
General & conductor & & $\varepsilon$ & 0 & \cite{SST} \\ 
 \hline
\end{tabular}
\caption{Known results toward the density conjecture of Katz and Sarnak}
\label{table:known-consequences}
\end{center}
\end{table}

The following table displays
\begin{itemize}
\item the family $\mathcal{F}$ under consideration, under the assumptions of Section \ref{sec:$L$-functions},
\item the aspect $c(\mathcal{F})$ used to order and truncate the family,
\item the type of symmetry $G$ of the family,
\item the limit $\delta$ allowed for the Fourier support in known results,
\item the value of $\eta(\mathcal{F}, \delta)$ obtained in in Theorem \ref{thm}, 
\item the reference in which the claimed results toward the density conjecture is proven.
\end{itemize}

The first line recovers the nonvanishing density from \cite{ILS} and the lower bound toward the distribution of central values of \cite{RS24}. The fourth line proves a statement announced, but unproven, in \cite{RS24} for modular forms. The third and fourth lines recover the nonvanishing densities in \cite{ILS}.

\subsection{Structure of the article and main ideas}
\label{subsec:intro-sketch}

To end this section, we introduce the structure of the rest of the paper. This also outlines the arguments to prove Theorem \ref{thm}.

Section \ref{sec:zeros} is dedicated to the proof of the first part of Theorem \ref{thm}, on the distribution of zeros.
The one-level density $D(L,\phi)$ is related to a sum over prime powers \textit{via} an explicit formula, which allows us to deduce the behavior of the low-lying zeros from the study of such arithmetic sums (Section \ref{subsec:llz-explicit-formula}). The large powers do not contribute thanks to the pointwise bounds on the coefficients (Section \ref{subsec:llz-high-powers}). The power~$2$, by means of the Hecke relations, relates to the coefficients of the $L$-function (which display cancellations and does not contribute significantly) up to the constant $\gamma_{\mathcal{F}}$, which critically contributes to the limiting behavior. The power $1$ may contribute when the summation is too long (Section \ref{subsec:llz-low-powers}). This limiting behavior of the one-level density on average classically allows one to deduce results on the nonvanishing at the central point (Section \ref{subsec:llz-nonvanishing}).

Section \ref{sec:l-values-and-zeros} and beyond are dedicated to the proof of the second part of Theorem \ref{thm}, on the distribution of central values.
The logarithm of $L(\tfrac12)$ is related to a sum over prime powers and a sum over zeros, analogously to the case of zeros, by means of an explicit formula (Section~\ref{subsec:values-explicit-formulas}). The sum over high powers does not contribute asymptotically (Section \ref{subsec:values-high-powers}). The sum over prime squares leads to possible biases from the Hecke relations, already critical in determining the type of symmetry of the family $\mathcal{F}$, and contributes to the mean $M_{\mathcal{F}}$ (Section \ref{subsec:values-low-powers}). The sum over primes results in the normal distribution, and this is proven by the method of moments, by showing that the moments of this sum over primes match the corresponding moments of the normal distribution, as stated in Proposition \ref{prop3}. 

The final result is then obtained by an amplification argument using the low-lying zeros behavior, proving that there are few elements $L \in \mathcal{F}$ such that the sum over zeros in the explicit formula has a large contribution, see Propositions \ref{prop4} and \ref{prop5}.
In particular, the remaining elements under consideration are the proportion $\eta(\mathcal{F}, \delta)$ occurring in Theorem \ref{thm}.

\section{From zeros to primes}
\label{sec:zeros}

We provide in this section the proof of the statement of Theorem \ref{thm} about low-lying zeros. This provides in particular a high-level sketch (technicalities being hidden in the assumptions made in Section \ref{sec:$L$-functions}) of all the proofs of the results toward Conjecture \ref{conj:katz-sarnak} in the literature.

\subsection{Explicit formula}
\label{subsec:llz-explicit-formula}

We state the Weil explicit formula \cite[(4.11)]{ILS}, which relates the one-level density $D(L,\phi)$ to a summation over primes of spectral parameters.
\begin{prop}\label{prop:explicit-formula-zeros}
We have, for every $L$-function $L(s)$ and for every Schwartz function $\phi$,
\begin{equation}
\label{eq:explicit-formula-zeros}
    D(L, \phi) := \sum_\gamma \phi(\tilde{\gamma}) = \widehat{\phi}(0) - \frac{2}{\log c(\mathcal{F})} \sum_{n \geqslant 1} \frac{\Lambda_L(n)}{n^{1/2}} \widehat{\phi}\left( \frac{\log n}{\log c(\mathcal{F})} \right) + O\left( \frac{1}{\log c(\mathcal{F})}\right), 
\end{equation}
where the summation runs over the renormalized zeros of $L(s)$.
\end{prop}

\subsection{High powers}
\label{subsec:llz-high-powers}

Using the Ramanujan bound $\alpha_{L,i}(p) \ll p^\varepsilon$, we obtain that the contribution of the powers $k \geqslant 3$ in \eqref{eq:explicit-formula-zeros} is
\begin{equation}
    \sum_p \sum_{k \geqslant 3} \frac{\Lambda_L(p^k)}{p^{k/2}} \widehat{\phi}\left( \frac{k \log p}{\log c(\mathcal{F})} \right) \ll \sum_{p} \sum_{k\geqslant 3} \frac{p^{k\varepsilon}}{p^{k/2}} \log p \ll \sum_p \frac{\log p}{p^{3/2-\varepsilon}} \ll 1.
\end{equation}
This implies that in \eqref{eq:explicit-formula-zeros}, only the prime powers $p^k$ with $k=1$ or $k=2$ may contribute to the limiting density. 

\begin{rk} \label{rk:alpha}
Weaker statements than Hypothesis \ref{Hyp:bound-spectral-parameters} of the form $\alpha_{L,i}(p) \ll p^{\alpha+\varepsilon}$ for a certain $\alpha>0$ are known unconditionally, e.g. we have $\alpha = 0$ for modular forms by Deligne \cite{deligne_formes_1973}, $\alpha = 7/64$ for Maass forms on $\GL(2)$ by Kim-Sarnak \cite{kim_refined_2003} and $\alpha = 1/2 - 1/(n+1)$ on $\GL(n)$ by Luo-Rudnick-Sarnak \cite{luo_generalized_1999}; see Blomer-Brumley \cite{blomer_ramanujan_2011} for further comments. The above argument remains valid for $\alpha < 1/6$, and would leave more values of $k$ to consider if $\alpha$ is larger. Note that if only $\alpha > 1/2$ is known (as for $\mathrm{GSp}(4)$ in the case of the standard $L$-function), then no term can be proved to be negligible with the above argument, and increasing $k$ a priori harms, e.g~\cite{comtat_effective_2025}.
\end{rk}

We are therefore left with
\begin{equation}
\label{eq:explicit-formula-zeros-2}
    D(L, \phi) \sim \widehat{\phi}(0) - \frac{2}{\log c(\mathcal{F})} \sum_p \sum_{k \in \{1, 2\}} \frac{\Lambda_L(p^k)}{p^{k/2}} \widehat{\phi}\left( \frac{k\log p}{\log c(\mathcal{F})} \right).
\end{equation}

\subsection{Low powers}
\label{subsec:llz-low-powers}

Using the Euler product \eqref{euler-product} and the Hecke relations from Hypothesis~\ref{Hyp:hecke-relations}, we have 
\begin{align}
\label{hecke1}
    \Lambda_L(p) & = \log(p) \sum_{i=1}^d \alpha_{L,i}(p) = a_L(p) \log(p), \\\label{hecke2}
    \Lambda_L(p^2) & = \log(p) \sum_{i=1}^d \alpha_{L,i}(p)^2 = (a_L(p^2) + \gamma_1 a_L(p) + \gamma_{\mathcal{F}})\log(p). 
\end{align}
Using the prime number theorem as in \cite[p. 1661]{MR4398244}, the factor $\gamma_{\mathcal{F}}$ from \eqref{hecke2} contributes
\begin{equation}
    2 \gamma_{\mathcal{F}} \sum_{p} \widehat{\phi}\left( \frac{2\log p}{\log c(\mathcal{F})} \right) \frac{\log p}{p \log c(\mathcal{F})} \sim \frac{1}{2}\gamma_{\mathcal{F}}\phi(0), 
\end{equation}
to the estimate~\eqref{eq:explicit-formula-zeros-2},
when $c(L)$ grows to infinity. It remains to estimate the contributions from the coefficients $a_L(p)$
and $a_L(p^2)$. On average over $L\in \mathcal{F}$, this is exactly explained by Hypothesis~\ref{Hyp:tf}, and the contribution therefore is $H_\phi = \int h_{\mathcal{F}}\phi$.
All in all, we obtain that
\begin{equation}
\label{eq:llz-result}
    \frac{1}{|\mathcal{F}|}\sum_{L \in \mathcal{F}} D(L,\phi) \sim \widehat{\phi}(0) - \frac{1}{2}\gamma_{\mathcal{F}}\phi(0)+ H_\phi = \int_{\mathbb{R}} W_{\mathcal{F}} \phi, 
\end{equation}
when $|\mathcal{F}| \to\infty$, where 
\begin{equation}
    W_{\mathcal{F}} = 1-\tfrac12\gamma_{\mathcal{F}}\delta_0 + h_{\mathcal{F}}, 
\end{equation}
as claimed in Theorem \ref{thm}.

Note in particular that the type of symmetry in the density statements is therefore determined by a common factor stemming from the Archimedean factors in the functional equation, by Hecke relations (as seen by the term $\gamma_{\mathcal{F}}$) and by a finer term that occurs when the summation over primes is long enough, viz. $H_\phi$ (which is the only one allowing to distinguish between the three orthogonal types of symmetry). 

\subsection{Nonvanishing consequences}
\label{subsec:llz-nonvanishing}

We follow closely \cite[(1.35) and below]{ILS} and introduce, for all $m \geqslant 0$, the density of the functions $L \in \mathcal{F}$ vanishing at the central point with order~$m$, viz.
\begin{equation}
    p_m := \frac{1}{|\mathcal{F}|} \# \{L \in \mathcal{F} \ : \ \underset{s=1/2}{\mathrm{ord}} L(s) = m\}.
\end{equation}
By definition, the sum of $p_m$ over $m \geqslant 0$ is equal to $1$. Morevover, for any nonnegative Schwartz function $\phi$ such that $\phi(0)=1$, we have that 
\begin{equation}
    \sum_{m \geqslant 1} mp_m \leqslant \frac{1}{|\mathcal{F}|}\sum_{L\in \mathcal{F}}D(L, \phi) \leqslant \int_{\mathbb{R}} W_{\mathcal{F}} \phi + \varepsilon
\end{equation}
for any $\varepsilon>0$, by \eqref{eq:llz-result}. In general, we can therefore deduce that 
\begin{equation}
    p_0 = 1-\sum_{m\geqslant 1} p_m \geqslant 1 - \sum_{m \geqslant 1} mp_m \geqslant 1 - \int_{\mathbb{R}} W_{\mathcal{F}} \phi - \varepsilon, 
\end{equation}
for all $\varepsilon>0$. An optimal choice of test function $\phi$ to maximize this quantity is given in \cite[Appendix A]{ILS} depending on the type of symmetry of the family and the allowed support $\delta$ for the Fourier transform. We obtain the values of $\eta(\mathcal{F}, \delta)$ in the first part of Table \ref{table:eta} and this proves the nonvanising consequence in Theorem \ref{thm}.

If the sign $\varepsilon_L$ in the functional equation is $1$ for all $L\in \mathcal{F}$, then taking the derivative of the functional equation~\eqref{functional-equation} and evaluating it at $1/2$ implies that $L'(1/2)=0$. Therefore, $p_1 = 0$ and any $L\in \mathcal{F}$ vanishing at $1/2$ vanishes with order at least $2$. The above argument therefore refines to
\begin{equation}
    p_0 = 1-\sum_{m\geqslant 2} p_m \geqslant 1 - \frac{1}{2}\sum_{m \geqslant 1} mp_m \geqslant 1 - \frac{1}{2}\int_{\mathbb{R}} W_{\mathcal{F}} \phi - \varepsilon, 
\end{equation}
for any $\varepsilon>0$. This, along with the optimal choice of function made in \cite[Appendix A]{ILS}, provides the values of $\eta(\mathcal{F}, \delta)$ in the second part of Table \ref{table:eta}.

\section{From \texorpdfstring{$L$}{L}-values to primes and zeros}
\label{sec:l-values-and-zeros}

We provide in this section the proof of the statement of Theorem \ref{thm} about central values under the assumptions of Section \ref{sec:$L$-functions}. This covers the cases treated in \cite{RS24} and \cite{LS}, and provides in particular a high-level sketch of these results.

\subsection{Explicit formula}
\label{subsec:values-explicit-formulas}

Our first proposition establishes an expression for the logarithmic central value in terms of sums over primes and over zeros, analogous to the explicit formula of Proposition \ref{prop:explicit-formula-zeros}. This is the analogue of \cite[Proposition~1]{RS24} or of \cite[Lemma 2]{BH}, and is the place where we need to assume the Generalized Riemann Hypothesis (we could assume weaker zero-free regions and have weaker error terms as in \cite[Lemma 2]{BH}). 

\begin{prop}\label{prop2}
For all $L$-function $L \in \mathcal{F}$ such that $L(\tfrac12)\neq 0$ and for $1<x<c(L)$, we have
\begin{equation}
\label{eq:explicit-formula-values}
\begin{aligned}
\log L\left(\frac12\right) &= \sum_{n \leqslant x} \frac{\Lambda_L(n) \log(x/n)}{n^{1/2} \log n \log x} + \frac{1}{\log x} \frac{L'}{L}\left( \frac12 \right) \\
    &\quad+ \frac{1}{\log x} \sum_\rho \int_{1/2}^\infty \frac{x^{\rho-\sigma}}{(\rho-\sigma)^2}d\sigma + O\left(\frac{1}{x^{1/2} \log^2 x}\right).
\end{aligned}
\end{equation}
\end{prop}

\subsection{High powers}
\label{subsec:values-high-powers}

Note that $L'/L(1/2) \ll \log c(L) \ll \log{c(\mathcal{F})}$, see \cite[(5.28) of Proposition 5.7]{IK04}. Using the pointwise bound on coefficients $a_L(n) \ll n^\varepsilon$, we deduce
\begin{equation}
    \sum_{k \geqslant 3}\sum_p \frac{\Lambda_L(p^k)}{p^{k/2} k\log p} \frac{\log(x/p^k)}{\log x} \ll 1.
\end{equation}
Therefore, analogously to Section \ref{subsec:llz-high-powers}, the only terms that may contribute to the size of $\log L(1/2)$ in \eqref{eq:explicit-formula-values} are those with $k=1$ and $k=2$. 

\subsection{Squares of primes}
\label{subsec:values-low-powers}

Analogously to the study of low-lying zeros in Section \ref{subsec:llz-low-powers}, the contributions of the low powers $k=1$ and $k=2$ are critical and determine the average and the distribution of the logarithm of central values.
Using \eqref{eq:rankin-selberg-on-average}, we deduce that the contribution of the squares $n=p^2$ to \eqref{eq:explicit-formula-values} is
\begin{equation}\label{squares_primes}
    \frac{1}{|\mathcal{F}|} \sum_{L\in \mathcal{F}}\sum_{p^2 \leqslant x} \frac{\Lambda_L(p^2)}{2p} \frac{\log(x/p^2)}{\log p \log x} \sim \frac{\gamma_{\mathcal{F}}}{2} \log \log x.
\end{equation}

Note that, on average over the family, this has the same distribution as the analogous quantity when removing the weight $\log(x/p)/\log x$ at a cost of $O(1)$. Indeed, the second moment of the difference is 
\begin{equation}
    \frac{1}{\log (x)^2} \frac{1}{|\mathcal{F}|}\sum_{L\in \mathcal{F}} \left| \sum_{p<x} \frac{a_L(p)\log{p}}{\sqrt{p}} \right|^2
\end{equation}
and, by the mean value theorem from Corollary \ref{coro:mean-value-theorem} and by partial summation, this is asymptotic to
\begin{equation}
    \frac{1}{\log (x)^2} \frac{1}{|\mathcal{F}|}\sum_{L\in \mathcal{F}} \sum_{p<x} \frac{a_L(p)^2 \log(p)^2}{p} \ll 1.
\end{equation}

We therefore obtain on averaging over $L \in \mathcal{F}$, for all $1<x<c(L)$, 
\begin{equation}
\label{aaa}
\begin{aligned}
\frac{1}{|\mathcal{F}|}\sum_{L \in \mathcal{F}} \log L(\tfrac12) & = \tfrac12\gamma_{\mathcal{F}} \log \log x + \frac{1}{|\mathcal{F}|} \sum_{L \in \mathcal{F}} \sum_{p < x} \frac{a_L(p)}{p^{1/2}} \\
& \qquad + O\left(\frac{\log c(\mathcal{F})}{\log x} + \frac{1}{|\mathcal{F}|}\sum_{L \in \mathcal{F}}\sum_\gamma \log(1 + (\gamma \log x)^{-2})\right),
\end{aligned}
\end{equation}
see, for example, \cite[Proof of Proposition~1, Section 4]{RS24} for more details about the proof.
Note that the mean $M_{\mathcal{F}}$ of the logarithms of central values appears as $\tfrac12 \gamma_{\mathcal{F}} \log \log x$, and it comes exactly from the contribution of the ``constant term" $\gamma_{\mathcal{F}}$, which reflects the critical contribution of the Hecke relations property to the result as in the case of the distribution of low-lying zeros. Unlike that case, however, the remaining terms remain to be studied, which then determine the distribution of these values.

\subsection{Moments of the sum of coefficients}
\label{subsec:values-moment-method}

To understand the distribution of the logarithmic central values $\log L(1/2)$, neglecting for now the error term \eqref{aaa} (it will be treated in Section \ref{subsec:amplifying-zeros}), we are led to explore the distribution of the sum over primes arising in \eqref{aaa}, i.e.
\begin{equation}
\label{eq:def-PL}
    P_L(x) := \sum_{p<x} \frac{a_L(p)}{p^{1/2}}
\end{equation}
when $x$ grows. To do so we apply the method of moments, and we give below the analogue of \cite[Proposition 3]{RS24}. It relates moments of the sum over primes $P_L(x)$ and moments of the normal distribution, given by
\begin{equation}\label{k-gaussian}
M_k := \frac{1}{\sqrt{2\pi}} \int_{-\infty}^\infty x^k e^{-x^2/2} dx = 
\begin{cases}
\displaystyle \frac{k!}{2^{k/2} (k/2)!}, & \text{if } k \in 2\mathbb{N}, \\[3mm]
0, &\text{otherwise.}
\end{cases}
\end{equation}

\begin{prop}\label{prop3}
We have, for all $k\geqslant 1$ and all $x$ bounded by a certain positive power (depending on $k$) of $|\mathcal{F}|$, 
\begin{equation}\label{prop3-1}
    \frac{1}{|\mathcal{F}|}\sum_{L \in \mathcal{F}} P_L(x)^k = \left(M_k+o(1)\right) (n_{\mathcal{F}}\log\log x)^{k/2}\,.
\end{equation}
Moreover, for all Schwartz function $\phi$ with a compactly supported Fourier transform, 
\begin{equation}\label{prop3-2}
    \frac{1}{|\mathcal{F}|}\sum_{L \in \mathcal{F}} P_L(x)^k D(L,\phi)= \left(M_k+o(1)\right) (n_{\mathcal{F}}\log\log x)^{k/2} \int_{\mathbb{R}} W_{\mathcal{F}} \phi.
\end{equation}
\end{prop}

\begin{proof}
Expanding the $k$-th power of $P_L(x)$ from the definition \eqref{eq:def-PL}, we get
\begin{equation}
\label{eq:expanded-moment}
    P_L(x)^k = \sum_{p_1, \ldots, p_k \leqslant x} \frac{a_L(p_1)\cdots a_L(p_k)}{\sqrt{p_1 \cdots p_k}}.
\end{equation}

A fundamental tool in this argument is the interplay of pointwise bounds on certain sums and average over $\mathcal{F}$. This relies on the following refinement of Hölder's inequality \cite[Lemmas A.1 and A.2 in Appendix A]{cheek_density_2024}, on which we build everytime we bound away portions of the average.
\begin{lem}
\label{lem:magic-lemma}
    For any $k \geqslant 1$ and $\Sigma_i(L) \in \mathbb{R}$ quantities depending on $L \in \mathcal{F}$, we have 
    \begin{equation}
        \frac{1}{|\mathcal{F}|}\sum_{L \in \mathcal{F}} \prod_{i=1}^k \left|\Sigma_i(L) \right| \leqslant \prod_{i=1}^k \left( \frac{1}{|\mathcal{F}|}\sum_{L\in\mathcal{F}} \Sigma_i(L)^{\xi_i}\right)^{1/\xi_i}
    \end{equation}
    for certain even integers $\xi_i \leqslant k+1$.
\end{lem}

This lemma is fundamental in our argument and is a tool to bound mixed moments inductively: the $\Sigma_i(L)$ will be various sums over primes of coefficients of $L$, and these will be bounded either by the pointwise Ramanujan (Hypothesis \ref{Hyp:bound-spectral-parameters}), by the orthogonality property (Hypothesis \ref{hyp:p-orthogonality}) or by certain averages over the family (Hypotheses \ref{Hyp:tf2} and \ref{Hyp:tf}). The above lemma then allows us to apply such bounds without paying the price of absolute values (since the $\xi_i$'s are even integers) and therefore bound inductively such quantities. Since the degree of the moment may increase by $1$ in Lemma \ref{lem:magic-lemma}, we have to ensure that at least two of the contributions of $\Sigma_i(L)$ are bounded away in the process before averaging (so that $k$ becomes $k-2$ and Lemma \ref{lem:magic-lemma} may worsen it to $k-1$).

We begin with the proof of \eqref{prop3-1}. 

\subsubsection{Induction and base case for \eqref{prop3-1}}

We use induction on the number $k$ of prime factors to prove the result. The idea has been streamlined in \cite[Lemma 6]{BH}, \cite[Proposition 3.1]{hough_distribution_2014} or in \cite[Theorem 3]{LS}, and we explain where the assumptions of Section \ref{sec:$L$-functions} enter into play. 

When $k=1$, the average of \eqref{eq:expanded-moment} is rewritten as
\begin{equation}
    \frac{1}{|\mathcal{F}|}\sum_{L\in \mathcal{F}}\sum_{p<x} \frac{a_L(p)}{\sqrt{p}} = \sum_{p<x} \frac{1}{\sqrt{p}} \frac{1}{|\mathcal{F}|}\sum_{L\in \mathcal{F}} a_L(p) \ll |\mathcal{F}|^{-\gamma} x^{1/2}
\end{equation}
by Hypothesis \ref{Hyp:tf2}. This is $o(\log\log(x)^{1/2})$ for $x < |\mathcal{F}|^{2\gamma}$. Moreover, for $k=2$, we have by using the mean value theorem from Corollary \ref{coro:mean-value-theorem}, 
\begin{equation}
    \frac{1}{|\mathcal{F}|}\sum_{L\in \mathcal{F}}\sum_{p_1, p_2<x} \frac{a_L(p_1)a_L(p_2)}{\sqrt{p_1p_2}} \sim \frac{1}{|\mathcal{F}|}\sum_{L\in \mathcal{F}} \sum_{p<x} \frac{a_L(p)^2}{p} \sim n_{\mathcal{F}} \log \log x, 
\end{equation}
by Hypothesis \ref{hyp:p-orthogonality}. We now assume $k \geqslant 2$, and split the summation into various different subsums according to whether or not the primes are equal or different.

\subsubsection{High powers for \eqref{prop3-1}}
\label{subsec:hi_pow}

If there is $m \geqslant 3$ such that $m$ primes among the $p_i$ are equal, then the sum is bounded by
\begin{equation}
\label{aaaa}
    \frac{1}{|\mathcal{F}|}\sum_{L \in \mathcal{F}} \left( \sum_{p<x} \frac{a_L(p)^m}{p^{m/2}}\right) \left( \sum_{\substack{p_i<x \\ p_i \neq p}} \frac{a_L(p_1) \cdots a_L(p_\ell)}{\sqrt{p_1\cdots p_\ell}}\right)
\end{equation}
for a certain $\ell \leqslant k-2$. 
The first sum is $O(1)$ by the bounds toward Ramanujan since $m\geqslant 3$, and this bound is uniform in $L$.
Therefore, by Lemma \ref{lem:magic-lemma}, we are reduced to bound the second sum which is $O(\log\log(x)^{\ell / 2}) = o(\log \log(x)^{k/2})$ by the induction hypothesis. 

Note that, in order to factor out the summation over $p$ in \eqref{aaaa}, we need to add back the missing terms $p_i = p$: these only add higher powers of primes (i.e. their contribution will be a quantity analogous to \eqref{aaaa} with a higher value of $m$) which therefore converge even more easily. Such a modification has therefore a negligible effect on the result, this argument was already steadily used in \cite{hough_distribution_2014, LS}.

\subsubsection{Power $2$ for \eqref{prop3-1}}
\label{xyz}

If exactly two primes are equal, then we can factor out 
\begin{equation}
\label{eq:power2}
    \sum_{p<x} \frac{a_L(p)^2}{p} = n_{\mathcal{F}} \log\log x + O(1)
\end{equation}
by Hypothesis \ref{hyp:p-orthogonality}, and use a similar argument to the one in Section \ref{subsec:hi_pow} to add back the missing terms when factoring out. The error term $O(1)$ is treated using Lemma \ref{lem:magic-lemma}, and the remaining sum having two less primes, we can conclude by induction. Critically, the main term $n_{\mathcal{F}}$ as well as the error term are assumed to be independent of the family $\mathcal{F}$ in Hypothesis \ref{hyp:p-orthogonality}, and by the induction hypothesis the remaining sum has contribution $M_{k-2}(n_{\mathcal{F}}\log\log x)^{((k-2)/2)}$, providing a contribution of $M_{k-2}(\log\log x)^{k/2}$ to the whole sum. There are $k-1$ possibilities to take pairs such two primes (we can fix $p_1$ as one of the primes since the indexation is an arbitrary choice, and then there are $k-1$ possibilities for the second $p_i$ taken to be equal to $p_1$), so that by the induction hypothesis applied to the remaining sum, the overall contribution is
\begin{equation}
    (k-1) \log\log(x) M_{k-2} \log\log(x)^{(k-2)/2} = M_k \log\log(x)^{k/2}, 
\end{equation}
as desired.

\subsubsection{Power $1$ for \eqref{prop3-1}}

If no primes are equal, then by multiplicativity of the coefficients we have $a_L(p_1)\cdots a_L(p_k) = a_L(p_1 \cdots p_k)$, and then the orthogonality relation from Hypothesis~\ref{Hyp:tf2} ensures
\begin{equation}
    \frac{1}{|\mathcal{F}|}\sum_{L \in \mathcal{F}} \sum_{p_i < x} \frac{a_L(p_1 \cdots p_k)}{\sqrt{p_1\cdots p_k}} = \sum_{p_i < x} \frac{1}{\sqrt{p_1\cdots p_k}} \frac{1}{|\mathcal{F}|}\sum_{L\in \mathcal{F}} a_L(p_1\cdots p_k) \ll |\mathcal{F}|^{-\gamma} x^{k/2}.
\end{equation}
This is $o(\log\log(x)^{k/2})$ as long as $x < |\mathcal{F}|^{2\gamma/k}$, so that the terms with power $1$ do not contribute to the limiting behavior in Proposition \ref{prop3} in this range. This concludes the proof by induction.

\subsubsection{Induction and base case for \eqref{prop3-2}}

The proof of \eqref{prop3-2}, with the extra weight $D(\phi, L)$, is similar by inputting the explicit formula for $D(\phi, L)$ from Proposition \ref{prop:explicit-formula-zeros}, so that we get
\begin{align}
    & \frac{1}{|\mathcal{F}|}\sum_{L \in \mathcal{F}} P_L(x)^k D(L,\phi) = \frac{1}{|\mathcal{F}|}\sum_{L \in \mathcal{F}} \sum_{p_i < x} \frac{a_L(p_1)\cdots a_L(p_k)}{\sqrt{p_1\cdots p_k}} \\
    & \qquad \times \left( \widehat{\phi}(0) - \frac{\gamma_{\mathcal{F}}}{2}\phi(0) + \sum_p \frac{a_L(p)}{p^{1/2}} \widehat{\phi}\left(\frac{\log p}{\log c(\mathcal{F})}\right) \frac{2\log p}{\log c(\mathcal{F})} + O\left( \frac{1}{\log c(\mathcal{F})}\right) \right). \notag
\end{align}
The part comprising $\hat{\phi}(0) - \tfrac12 \gamma_{\mathcal{F}}\phi(0)$ does not depend on $L$ and can be factored out, so that the corresponding contribution is exactly
\begin{equation}
\label{MT}
    \left( \widehat{\phi}(0) - \frac{\gamma_{\mathcal{F}}}{2}\phi(0) \right) \frac{1}{|\mathcal{F}|}\sum_{L \in \mathcal{F}} P_L(x)^k
\end{equation}
and, by the first part of Proposition \ref{prop3} proved above, this is $(\hat{\phi}(0) - \tfrac12\gamma_{\mathcal{F}}\phi(0))M_k\log\log(x)^{k/2} + o(1)$. The remaining contribution, that features the coefficients of $L$, is
\begin{equation}
    \frac{1}{|\mathcal{F}|}\sum_{L \in \mathcal{F}} \sum_{p_i < x,\, p} \frac{a_L(p_1)\cdots a_L(p_k)}{\sqrt{p_1\cdots p_k}}\frac{a_L(p)}{\sqrt{p}} \widehat{\phi}\left(\frac{\log p}{\log c(\mathcal{F})}\right) \frac{2\log p}{\log c(\mathcal{F})}.
\end{equation}
The proof follows the lines of the first case of the proposition and proceeds by induction over~$k$.

When $k=1$, we separate the summation depending on whether or not $p$ and $p_1$ are equal. When they are not equal, we use the multiplicativity of the coefficients so that the corresponding summation is 
\begin{equation}
    \sum_{\substack{p_1 < x,\, p \\ p \neq p_1}} \frac{1}{\sqrt{pp_1}} \widehat{\phi}\left(\frac{\log p}{\log c(\mathcal{F})}\right) \frac{2\log p}{\log c(\mathcal{F})}\frac{1}{|\mathcal{F}|}\sum_{L\in \mathcal{F}} a_L(pp_1).
\end{equation}
By Hypothesis \ref{Hyp:tf2}, the inner sum is bounded by $|\mathcal{F}|^{-{\gamma}}$ and then this quantity is bounded by $|\mathcal{F}|^{-{\gamma}}x/\log c(\mathcal{F})$, which is $o(\log\log x)$ as soon as $x < |\mathcal{F}|^{\gamma}$.

When $p_1 = p$, then the summation is 
\begin{equation}
    \frac{1}{|\mathcal{F}|}\sum_{L\in \mathcal{F}} \sum_{p<x} \frac{a_L(p)^2}{p} \widehat{\phi}\left(\frac{\log p}{\log c(\mathcal{F})}\right)\frac{2\log p}{\log c(\mathcal{F})}.
\end{equation}
By Hypothesis \ref{hyp:p-orthogonality} and summation by parts, we obtain that this is $O(1) = o(\log \log(x)^{1/2})$ as desired.

\subsubsection{High powers for \eqref{prop3-2}}

If $m \geqslant 3$ primes are equal, the sum is either
\begin{equation}
    \frac{1}{|\mathcal{F}|}\sum_{L \in \mathcal{F}} \left(\sum_{p_1<x} \frac{a_L(p_1)^m}{p_1^{m/2}}\right) \sum_{p_i<x,\, p}\frac{a_L(p_2)\cdots a_L(p_\ell)a_L(p)}{\sqrt{p_2\cdots p_\ell p}}\widehat{\phi}\left( \frac{\log p}{\log c(\mathcal{F})}\right)\frac{\log p}{\log c(\mathcal{F})}, 
\end{equation}
or, if the prime $p$ from the one-level density is among them, 
\begin{equation}
    \frac{1}{|\mathcal{F}|}\sum_{L \in \mathcal{F}} \left(\sum_{p<x} \frac{a_L(p)^m}{p^{m/2}} \widehat{\phi}\left( \frac{\log p}{\log c(\mathcal{F})}\right)\frac{\log p}{\log c(\mathcal{F})} \right) \sum_{p_i<x}\frac{a_L(p_2)\cdots a_L(p_\ell)}{\sqrt{p_2\cdots p_\ell}}, 
\end{equation}
where the second sum has been completed by those $p_i = p_1$ or $p_i = p$ respectively, as in the previous section, at the cost of an error term.
The expression between parentheses are $O(1)$ in both cases since $a_L(p) \ll p^\varepsilon$ uniformly in $L$. Lemma \ref{lem:magic-lemma} and the induction hypothesis then ensure that the remaining sum is $O(\log \log(x)^{k/2-1/2}) = o(\log\log(x)^k)$.

\subsubsection{Power $2$ for \eqref{prop3-2}}

If exactly two primes are paired, then one possibility is that they are among the $p_i$'s and the expression therefore is 
\begin{equation}
    \frac{1}{|\mathcal{F}|}\sum_{L \in \mathcal{F}} \left(\sum_{p_1<x} \frac{a_L(p_1)^2}{p_1} \right) \sum_{p_i<x,\, p}\frac{a_L(p_2)\cdots a_L(p_\ell)a_L(p)}{\sqrt{p_2\cdots p_\ell p}}\widehat{\phi}\left( \frac{\log p}{\log c(\mathcal{F})}\right)\frac{\log p}{\log c(\mathcal{F})}, 
\end{equation}
and the bracketed expression is $n_{\mathcal{F}}\log\log x + O(1)$ by Hypothesis \ref{hyp:p-orthogonality}. The error term is bounded away by Lemma \ref{lem:magic-lemma} and the induction hypothesis. The term in $\log \log x$ provides the only one contributing as a main term, and there are $k-1$ such pairing possibilities (for the same reasons as in Section \ref{xyz}), so that by the induction hypothesis applied to the remaining sum, the overall contribution is
\begin{equation}
    (k-1) \log\log(x) H_\phi M_{k-2} \log\log(x)^{(k-2)/2} = M_k \log\log(x)^{k/2}, 
\end{equation}
as desired.

The other possibility is that the two primes that are equal feature the prime $p$ coming from the one-level density and the expression is 
\begin{equation}
    \frac{1}{|\mathcal{F}|}\sum_{L \in \mathcal{F}} \left(\sum_{p<x} \frac{a_L(p)^2}{p} \widehat{\phi}\left( \frac{\log p}{\log c(\mathcal{F})}\right)\frac{\log p}{\log c(\mathcal{F})} \right) \sum_{p_i<x,\, p}\frac{a_L(p_2)\cdots a_L(p_\ell)}{\sqrt{p_2\cdots p_\ell}}, 
\end{equation}
The bracketed expression is $O(1)$ by Hypothesis $\ref{Hyp:tf}$, and Lemma \ref{lem:magic-lemma} therefore concludes the proof by induction.

\subsubsection{Power $1$ for \eqref{prop3-2}}

If no prime is paired, then by multiplicativity the coefficient $a_L(p_1\cdots p_kp)$ occurs alone and the summation, which looks like
\begin{equation}
    \sum_{p_i < x,\, p} \frac{1}{\sqrt{p_1\cdots p_k p}} \widehat{\phi}\left(\frac{\log p}{\log c(\mathcal{F})}\right) \frac{2\log p}{\log c(\mathcal{F})} \frac{1}{|\mathcal{F}|}\sum_{L \in \mathcal{F}} a_L(p_1 \cdots p_k p).
\end{equation}
The innermost sum is bounded by $|\mathcal{F}|^{-\gamma}$ by Hypothesis \ref{Hyp:tf}, so that the whole sum is bounded by $x^{(k+1)/2}|\mathcal{F}|^{-\gamma}$, which is $O(1) = o(\log\log(x)^{k/2})$ for $x < |\mathcal{F}|^{2\gamma/(k+1)}$.

This concludes the proof of Proposition \ref{prop3} by induction.
\end{proof}

\subsection{Moment method}

The following result, analogous to \cite[Lemma 1]{RS24} or \cite[Proposition 3.1]{hough_distribution_2014}, uses the ``moment method" to quantify the proportion of $L \in \mathcal{F}$ such that $P_L(x)$ falls into a specific interval.
\begin{prop}
\label{coro:moments}
We have, for all Schwartz function $\phi$ with compactly supported Fourier transform, 
\begin{equation}
    \sum_{\substack{L \in \mathcal{F} \\ P_L(x)/\sqrt{n_{\mathcal{F}}\log\log x} \in (\alpha, \beta)}} D(L, \phi) = (M(\alpha, \beta) + o(1))\sum_{L \in \mathcal{F}} D(L, \phi),
\end{equation}
where 
\begin{equation}\label{eq:M_alpha_beta}
    M(\alpha, \beta) := \frac{1}{\sqrt{2\pi}}\int_\alpha^\beta e^{-x^2/2} dx.
\end{equation}
\end{prop}

\begin{proof}
Asymptotically, Proposition \ref{prop3} shows that the $\ell$-th moment of $P_L(x)/\sqrt{n_{\mathcal{F}}\log\log x}$ behaves as the $\ell$-th moment of the normal distribution, i.e. when $x$ grows to infinity,
\begin{equation}
    \sum_{L \in \mathcal{F}} \left( \frac{P_L(x)}{\sqrt{n_{\mathcal{F}}\log\log x}} \right)^\ell D(L, \phi) \sim \frac{1}{\sqrt{2\pi}}\int_{\mathbb{R}} x^\ell e^{-x^2/2} dx \int_{\mathbb{R}} W_{\mathcal{F}}\phi,
\end{equation}
for all $\ell \geqslant 0$, so we deduce that, for any polynomial $R \in \mathbb{R}[X]$, 
\begin{equation}
    \frac{1}{|\mathcal{F}|} \sum_{L \in \mathcal{F}} R\left( \frac{P_L(x)}{\sqrt{n_{\mathcal{F}}\log\log x}} \right) D(L, \phi) \sim \frac{1}{\sqrt{2\pi}}\int_{\mathbb{R}} R(x) e^{-x^2/2} dx \int_{\mathbb{R}} W_{\mathcal{F}}\phi,
\end{equation}
and by approximating the characteristic function $\mathbf{1}_{(\alpha, \beta)}$ by a polynomial $R$, we deduce that 
\begin{align}
    \frac{1}{|\mathcal{F}|}\sum_{\substack{f \in \mathcal{F} \\ P_L(x)/\sqrt{n_{\mathcal{F}}\log\log x} \in (\alpha, \beta)}} D(L, \phi) & 
    = \sum_{f \in \mathcal{F}} \mathbf{1}_{(\alpha, \beta)}\left( \frac{P_L(x)}{\sqrt{n_{\mathcal{F}}\log\log x}} \right) D(L, \phi) \\
    & \sim \frac{1}{\sqrt{2\pi}} \int_{\mathbb{R}} \mathbf{1}_{(\alpha, \beta)}(x) e^{-x^2/2} dx \int_{\mathbb{R}} W_{\mathcal{F}}\phi = M(\alpha, \beta) \int_{\mathbb{R}} W_{\mathcal{F}}\phi,
\end{align}
as claimed.
\end{proof}

\section{Endgame}

\subsection{Amplifying zeros}
\label{subsec:amplifying-zeros}

Analogously to \cite{RS24}, we use the previous results to explain that, even if there may be $L$-functions in the family $\mathcal{F}$ having small zeros such that the error term in Proposition \ref{prop2} is harmful, there is only a very small proportion of them in the family. This can be controlled via the density results obtained in Section \ref{sec:zeros} which provide information about the low-lying zeros.

\begin{prop}\label{prop4}
For all $\alpha < \beta$, the number of $L\in \mathcal{F}$ such that $P_L(x) / \sqrt{n_{\mathcal{F}}\log \log x} \in (\alpha, \beta)$ and such that there are no zeros with $|\gamma_L| \leqslant (\log X \log\log X)^{-1}$ (where $x=X^{1/\log \log \log X}$) is larger than 
\begin{equation}
    \eta(\mathcal{F}, \delta) M(\alpha, \beta) |\mathcal{F}|
\end{equation}
where 
\begin{equation}
    M(\alpha, \beta) = \frac{1}{\sqrt{2\pi}}\int_\alpha^\beta e^{-x^2/2} dx.
\end{equation}
\end{prop}

\begin{rk}
Note that this is where Bombieri-Hejhal \cite{BH} and Hough \cite{hough_distribution_2014,hough_distribution_2014err} get a direct conclusion using a zero density hypothesis, ensuring that this error term does not contribute. Under the assumption toward the Katz-Sarnak conjecture in place of zero density hypothesis, this cannot be ensured and this is where a certain proportion of $L$-functions in the family is lost.
\end{rk}

\begin{proof}
Choose for $\phi$ the explicit Féjer kernel up to the maximal Fourier support~$\delta$ allowed by the low-lying zero assumption. More precisely, let
\begin{equation}
    \phi_0(x) := \left( \frac{\sin \pi x}{\pi x}\right)^2, \qquad \widehat{\phi}_0(y) = \max(1-|y|, 0), 
\end{equation}
so that $\widehat{\phi}_0(y)$ is supported in $(-1,1)$,
and $\phi(x) = \phi_0(x\delta)$ so that $\widehat{\phi}(y) = \delta^{-1}\widehat{\phi}_0(x/\delta)$ is compactly supported in $(-\delta, \delta)$. Proposition \ref{coro:moments} implies
\begin{align}
    \frac{1}{|\mathcal{F}|} \sum_{\substack{f \in \mathcal{F} \\ P_L(x)/\sqrt{\log\log x} \in (\alpha, \beta)}} D(L, \phi) & \sim M(\alpha, \beta) \frac{1}{|\mathcal{F}|}\sum_{f \in \mathcal{F}} D(L, \phi),
\end{align}
and, by the low-lying zero limiting density \eqref{Hyp:weighted-llz}, we deduce
\begin{align}
      M(\alpha, \beta) \frac{1}{|\mathcal{F}|}\sum_{f \in \mathcal{F}} D(L, \phi) \sim M(\alpha, \beta) \int_{\mathbb{R}} W_{\mathcal{F}}\phi.
\end{align}
Iwaniec, Luo and Sarnak \cite[(1.42)--(1.46)]{ILS} computed the values of $\kappa := \int W_{\mathcal{F}}\phi$ for this specific choice of $\phi$ and gave its dependency in $\delta$.

\begin{center}
\begin{table}[ht]
\begin{tabular}{ |c|c|c|c| } 
 \hline
 Type & $\delta<1$ & $\delta \geqslant 1$ & $\delta=\infty$ \\ 
 \hline
 O & $\tfrac{1}{\delta} + \tfrac{1}{2}$ & $\tfrac{1}{\delta} + \tfrac{1}{2}$ & $\tfrac12$ \\ 
 \hline
  SO(even) & $\tfrac{1}{\delta} + \tfrac{1}{2}$ & $\tfrac{2}{\delta} - \tfrac{1}{2\delta^2}$ & $0$ \\ 
 \hline
  SO(odd) & $\tfrac{1}{\delta} + \tfrac{1}{2}$ & $1 + \tfrac{1}{2\delta^2}$ & $1$ \\ 
 \hline
    Sp & $\tfrac{1}{\delta} - \tfrac{1}{2}$ & $\tfrac{1}{2\delta^2}$ & $0$ \\ 
 \hline
    U & $\tfrac{1}{\delta}$ & $\tfrac{1}{\delta}$ & $0$ \\
 \hline
\end{tabular}
\caption{Values of $\kappa = \int W_{\mathcal{F}}\phi$ depending on $\delta$ and the type of symmetry}
 \label{table:kappa}
\end{table}
\end{center}

\vspace{-1.6cm}

We can now use the similar amplification argument as in \cite{RS24} or \cite{LS}. Let $\Phi := D(L, \phi)$ to lighten notation. Split the above sum depending on whether or not the $L$-function has too small zeros: 
\begin{equation}
  \sum_{(\alpha, \beta)} \Phi = \sum_{\substack{(\alpha, \beta) \\ \exists}} \Phi + \sum_{{\substack{(\alpha, \beta) \\ \nexists}}} \Phi 
\end{equation}
where
\begin{align}
    \sum_{(\alpha, \beta)} \Phi & := \sum_{\substack{L \in \mathcal{F} \\ P_L(x)/\sqrt{n_{\mathcal{F}}\log\log x} \in (\alpha, \beta)}} \Phi, \\
    \sum_{\substack{(\alpha, \beta) \\ \exists}} \Phi & := \sum_{\substack{L \in \mathcal{F} \\ \exists |\gamma_L| \leqslant \ell\\ P_L(x)/\sqrt{n_{\mathcal{F}}\log\log x} \in (\alpha, \beta)}} \Phi, \\
    \sum_{{\substack{(\alpha, \beta) \\ \nexists}}} \Phi & := \sum_{\substack{L \in \mathcal{F} \\ \nexists |\gamma_L| \leqslant \ell\\ P_L(x)/\sqrt{n_{\mathcal{F}}\log\log x} \in (\alpha, \beta)}} \Phi,
\end{align}
and $\ell = (\log X \log\log X)^{-1}$.
The weights $ \phi(\gamma)$ in $\Phi$ are always non-negative, since the function $\phi$ we chose is non-negative. Since $\ell \to 0$ when $x$ grows, we have $\phi(\gamma) \geqslant \phi(0) - \varepsilon = 1 - \varepsilon$. We can therefore write
\begin{equation}
      \sum_{(\alpha, \beta)} \Phi = \sum_{\substack{(\alpha, \beta) \\ \exists}} \Phi + \sum_{{\substack{(\alpha, \beta) \\ \nexists}}} \Phi \geqslant \sum_{\substack{(\alpha, \beta) \\ \exists}} 1 + \sum_{{\substack{(\alpha, \beta) \\ \nexists}}} \Phi = \sum_{(\alpha, \beta)} 1 + \sum_{{\substack{(\alpha, \beta) \\ \nexists}}} (\Phi-1) .
\end{equation}
On the other hand, the above consequence of the moment method and of the limiting one-level density result are rephrased as 
\begin{align}
    \sum_{(\alpha, \beta)} 1 \sim M(\alpha, \beta) |\mathcal{F}| \quad \text{and} \quad \sum_{(\alpha, \beta)} \Phi \sim M(\alpha, \beta) \kappa |\mathcal{F}|, 
\end{align}
relating the restricted sums to the corresponding whole sums. We therefore deduce, since $0 \leqslant \phi \leqslant 1$, 
\begin{equation}
    \sum_{\substack{(\alpha, \beta) \\ \nexists}} 1 \geqslant \sum_{\substack{(\alpha, \beta) \\ \nexists}} (1-\Phi) \geqslant \sum_{(\alpha, \beta)} 1 - \sum_{(\alpha, \beta)} \Phi \sim M(\alpha, \beta)(1-\kappa)|\mathcal{F}|.
\end{equation}
We can then lower bound the smoothed quantity of $f \in \mathcal{F}$ having zeros of moduli smaller than $\ell$, viz.
\begin{equation}
    \sum_{\substack{(\alpha, \beta) \\ \nexists}} 1 \geqslant (1-\kappa) M(\alpha, \beta) |\mathcal{F}|,
\end{equation}
as desired, giving the values of $\eta(\mathcal{F}, \delta) := 1-\kappa$ recorded in Table \ref{table:eta}.

If the $L$-functions in $\mathcal{F}$ all have signs $\varepsilon_L = 1$ in the functional equation, then each small zeros arise in pair: either the zero is at the central point i.e. $\gamma = 0$ and then it is necessarily a double zero as explained in Section \ref{subsec:llz-nonvanishing}; or it is nonzero, and its conjugate has the same modulus, in particular it is also small. The above argument then carries on and gives the improved values of $\eta(\mathcal{F}, \delta) = 1 - \tfrac{\kappa}{2}$ stated in Table \ref{table:eta}. In both cases, it is the same value as the lower bound for the density of nonvanishing obtained in the first part of Theorem \ref{thm}.
\end{proof}

\subsection{Few zeros contributing a lot}

The following proposition is the analogue of \cite[Lemma~2]{RS24}, and quantifies how rare are the $L \in \mathcal{F}$ such that the contribution from the sum over zeros in Proposition \ref{prop2} is large.

\begin{prop}\label{prop5}
For all $X\gg 1$, the number of $L \in \mathcal{F}$ such that 
\begin{equation}
    \sum_{|\gamma| \geqslant (\log X \log\log X)^{-1}} \log(1+(\gamma_L \log x)^{-2}) \geqslant \log \log \log (X)^3 
\end{equation}
is $\ll |\mathcal{F}|/\log \log \log X$.
\end{prop}

\begin{proof}
The same proof as in \cite[Lemma 2]{RS24} holds \textit{mutatis mutandis}.
\end{proof}

\subsection{Proof of the theorem}

This closely follows \cite{RS24} final arguments, now that all the corresponding estimates have been established. We write it here for the sake of completeness. Using the notation \eqref{eq:def-PL}, recall from \eqref{aaa} that
\begin{equation*}
\begin{aligned}
    \frac{1}{|\mathcal{F}|}\sum_{L\in \mathcal{F}}\log L\left(\tfrac{1}{2}\right) = \frac{\gamma_{\mathcal{F}}}{2}\log \log x &+ \frac{1}{|\mathcal{F}|}\sum_{L\in \mathcal{F}} P_L(x) \\
    &+ O\left(\frac{\log c(\mathcal{F})}{\log x} + \frac{1}{|\mathcal{F}|}\sum_{L\in \mathcal{F}} \sum_\gamma \log\left(1 + (\gamma \log x)^{-2}\right)\right).
\end{aligned}
\end{equation*}

By Proposition \ref{prop4}, there is a proportion at least $\eta(\mathcal{F}, \delta)M(\alpha, \beta)$ of $L \in \mathcal{F}$ such that
$$ \frac{P_L(x)}{\sqrt{n_{\mathcal{F}} \log \log x}} \in (\alpha, \beta) $$
and $L(s)$ has no zeros such that $|\gamma| \leqslant (\log X \log \log X)^{-1}$. By Proposition \ref{prop5}, we can remove $|\mathcal{F}|/\log \log \log X = o(|\mathcal{F}|)$ of such elements to ensure that the contribution of zeros $|\gamma| \geqslant (\log X \log \log X)^{-1}$ is $O\left((\log \log \log X)^3\right)$. All in all, there are at least $\eta(\mathcal{F}, \delta)M(\alpha, \beta)|\mathcal{F}|$ elements $L\in \mathcal{F}$ such that 
\begin{equation}
    \frac{\log L(\tfrac12) - \tfrac12 \gamma_{\mathcal{F}} \log \log c(\mathcal{F})}{\sqrt{n_{\mathcal{F}} \log \log c(\mathcal{F})}} + O\left( \frac{(\log \log \log |F|)^3}{\sqrt{\log \log c(\mathcal{F})}}\right) \in (\alpha, \beta).
\end{equation}
This concludes the proof of Theorem \ref{thm}.

\subsection*{Acknowledgment} We are grateful to Youness Lamzouri, Maksym Radziwi\l{}\l{} and Kannan Soundararajan for enlightening discussions. 
This work started when D. L. was visiting Kyushu University, continued while A. I. S. was visiting Universit\'e de Lille, and ended when both authors were visiting San Jose State University; we thank all these institutions for very good working environments. We acknowledge support from the R-CDP-24-004-C2EMPI project, the~CNRS (IEA), JSPS KAKENHI Grant Number 22K13895, and Kyushu University International Research Leader Training Program (EBXU0101).

\bibliographystyle{alpha}
\bibliography{LLZBibliography}

\end{document}